\documentclass[12pt,a4paper,pdftex]{amsart}
\pdfoutput=1
\usepackage{amssymb}
\usepackage{esint}
\usepackage{mathtools}
\usepackage{hyperref}
\mathtoolsset{showonlyrefs}
\usepackage[english]{babel}
\usepackage[latin1]{inputenc}
\usepackage{graphicx}
\theoremstyle{plain}
\newtheorem{theorem}{Theorem}
\newtheorem{lemma}[theorem]{Lemma}
\newtheorem{corollary}[theorem]{Corollary}
\newtheorem{proposition}[theorem]{Proposition}

\theoremstyle{definition}
\newtheorem{definition}[theorem]{Definition}

\theoremstyle{remark}
\newtheorem{remark}[theorem]{Remark}
\newcommand{\R}{\mathbb R}

\newcommand{\N}{\mathbb N}

\newcommand{\der}{\mathrm{d}}
\newcommand{\eps}{\varepsilon}
\renewcommand{\phi}{\varphi}
\newcommand{\abs}[1]{\left| #1 \right|}
\newcommand{\aabs}[1]{\left\| #1 \right\|}
\newcommand{\sisus}{\operatorname{int}}

\renewcommand{\theta}{\vartheta}

\newcommand{\ext}[1]{\widehat{#1}}
\newcommand{\id}{\operatorname{id}}

\newcommand{\ip}[2]{\left\langle#1,#2\right\rangle}

\newcommand{\Der}[1]{\frac{\der}{\der #1}}
\title[Broken ray transform on a Riemann surface]{Broken ray transform on a Riemann surface with a convex obstacle}
\author{Joonas Ilmavirta}
\address{Department of Mathematics and Statistics, University of Jyv\"askyl\"a}
\email{joonas.ilmavirta@jyu.fi}
\author{Mikko Salo}
\address{Department of Mathematics and Statistics, University of Jyv\"askyl\"a}
\email{mikko.j.salo@jyu.fi}
\date{\today}

\begin{document}

\begin{abstract}
We consider the broken ray transform on Riemann surfaces in the presence of an obstacle, following earlier work of Mukhometov \cite{M4}. If the surface has nonpositive curvature and the obstacle is strictly convex, we show that a function is determined by its integrals over broken geodesic rays that reflect on the boundary of the obstacle. Our proof is based on a Pestov identity with boundary terms, and it involves Jacobi fields on broken rays. We also discuss applications of the broken ray transform.
\end{abstract}

\maketitle

\section{Introduction}
\label{sec:intro}

\subsection{Basic setup}
This article considers X-ray transforms in domains with obstacles.
A basic setting would be a domain $M = \overline{B} \setminus \cup_{j=1}^N \sisus K_j$, where $B \subset \R^2$ is an open ball and $K_j \subset B$ are compact pairwise disjoint obstacles.
We assume that $f: M \to \R$ is an unknown continuous function, and that we can obtain X-ray measurements for~$f$ on~$\partial B$ but not on the boundaries of the obstacles~$K_j$.

Let us first assume that we know the integrals of~$f$ over all line segments in~$M$ that do not touch~$C$, where~$C$ is the convex hull of $\cup_{j=1}^N K_j$. Then the Helgason support theorem~\cite{book-helgason} implies that $f|_{M \setminus C}$ is uniquely determined by these integrals. Clearly these integrals do not contain any information about $f$ in $C \setminus \cup_{j=1}^N K_j$, and more information is needed to determine~$f$ in this set.

In certain applications, such as inverse problems for Schr\"odinger equations~\cite{eskin} or the Calder\'on problem~\cite{KS:calderon} with partial data, one has knowledge of integrals of~$f$ over all \emph{broken rays} in~$M$ that start and end on~$\partial B$ and reflect on the boundaries of the obstacles according to geometrical optics. In particular, one knows the integrals of~$f$ over all line segments that do not touch~$C$ as above, but one additionally has information on~$f$ along broken rays that reach the set $C \setminus \cup_{j=1}^N K_j$. One could try to use this additional information to determine~$f$ in all of~$M$.

There are easy counterexamples showing that if the obstacles are not strictly convex, then it is not possible to determine~$f$ everywhere (see Figures~\ref{fig:ctr-ex1} and~\ref{fig:ctr-ex2}). It is therefore natural to consider the case of strictly convex obstacles. If there are at least two obstacles, one will always have broken rays with infinite length or with multiple tangential reflections that will make the analysis considerably more complicated. For this reason, we  begin our study of this problem by considering the case of only one reflecting obstacle in this paper.

Our main result is a uniqueness theorem for the broken ray transform in a compact nonpositively curved Riemannian manifold $(M,g)$ with boundary, so that~$M$ contains one strictly convex reflecting obstacle. As described above, such a uniqueness result follows whenever a Helgason support theorem is available. Besides in Euclidean space, Helgason support theorems are known on simple manifolds of dimension $\geq 2$ with real-analytic metric~\cite{K:spt-thm}, and on manifolds of dimension $\geq 3$ having a suitable foliation by convex hypersurfaces~\cite{UV:local-x-ray}. Our theorem below contains the case of nonpositively curved 2D manifolds with smooth (not necessarily real-analytic) metric, and is therefore not a consequence of any known Helgason support theorem. More importantly, the proof involves a general PDE method based on energy estimates (Pestov identity with boundary terms) that may extend to the case of several obstacles.

The PDE approach to integral geometry problems with reflected rays is due to by Mukhometov \cite{M1, M2, M3, M4, M5, M6}, following his earlier work for non-reflected rays \cite{M:bdy-rig-surface-eng}. In particular, \cite{M4} states a theorem (without proof) that is very close to our main result below, even for certain weights and including a stability estimate. The same approach was used by Eskin~\cite{eskin} in~$\R^2$ to obtain a result for several obstacles under additional restrictions, including that the obstacles necessarily contain corner points.

In this paper we also follow the general approach of Mukhometov. A major motivation for this paper is to prepare for the possible treatment of several convex obstacles. In view of this we discuss the regularity of solutions and Jacobi fields on broken rays in some detail, and also give a convenient proof of the relevant energy estimate in the spirit of \cite{PSU:tensor-surface}.

\subsection{Main result}
Let~$(M,g)$ be an orientable, compact smooth Riemannian surface with smooth boundary. We make the assumption that $(M,g)$ is contained in some disc $(D,g)$ in $\R^2$, which will ensure that there are global isothermal coordinates on~$M$. 
Suppose~$\partial M$ is composed of two disjoint parts~$E$ and~$R$ which are unions of connected components but need not be connected.
Below we will think of the manifold as a larger manifold~$\ext{M}$ from which an (open) obstacle $O\subset\ext{M}$ has been removed.
In this setting $M=\ext{M}\setminus O$, $E=\partial\ext{M}$ and~$R=\partial O$; in this sense~$R$ is the inner and~$E$ the outer boundary of~$M$.

We say that a curve~$\gamma$ on~$M$ is a \emph{broken ray} if it is geodesic in~$\sisus M$ and reflects on~$R$ according to the usual reflection law: the angle of incidence equals the angle of reflection.
The \emph{broken ray transform} of a function $f:M\to\R$ is the map that takes a broken ray with both endpoints in~$E$ into the integral of~$f$ over the broken ray.

We denote the unit sphere bundle of~$M$ by $SM=\bigcup_{x\in M}S_x$, $S_x=\{v\in T_xM;\abs{v}=1\}$.
Let~$\nu$ be the outer unit normal at~$\partial M$.
We write $E_\pm=\{(x,v)\in SM;x\in E,\pm v\cdot\nu>0\}$, $E_0=\{(x,v)\in SM;x\in E,v\cdot\nu=0\}$, and $SE=E_+\cup E_0\cup E_-$.
We define~$R_\pm$, $R_0$ and~$SR$ similarly, and note that $\partial SM=SE\cup SR$.

For any $(x,v)\in SM\setminus SE$ there is a unique broken ray $\gamma_{x,v}^E:[0,T]\to M$ with $\gamma_{x,v}^E(0)=x$, $\dot{\gamma}_{x,v}^E(0)=v$ for maximal~$T$ such that the broken ray remains in~$M$.
We let $\tau_{x,v}^E=\min\{t\geq0;\gamma_{x,v}^E(t)\in E\}$ be the exit time.
It may happen that $\gamma_{x,v}^E(t)\in M\setminus E$ for all $t\geq0$; in this case $\tau_{x,v}^E=\infty$.
If $(x,v)\in SE$ we define $\tau_{x,v}^E$ as above if $v\cdot\nu \leq 0$, and $\tau_{x,v}^E = 0$ otherwise.

\begin{theorem}
\label{thm:brt}
The broken ray transform is injective on $C^2(M)$ if the following conditions hold:
\begin{enumerate}
\item\label{cond:E-convex} $E$ is strictly convex,
\item\label{cond:curv} the Gaussian curvature of~$M$ is nonpositive,
\item\label{cond:1tang-hit} there is a number $a\in(0,1]$ such that every broken ray has at most one reflection with $\abs{\ip{\dot\gamma}{\nu}}<a$,
 and
\item\label{cond:obs} there is a constant~$L$ such that for any $(x,v)\in\sisus SM$ we have $\tau_{x,v}^E\leq L$.
\end{enumerate}
\end{theorem}

\begin{remark}
\label{rmk:convex}
Condition~\ref{cond:1tang-hit} implies that~$R$ is concave (or the obstacle is convex). For if the obstacle were not convex, there would be a point on~$R$ where~$R$ is strictly convex.
Near such a point one can easily construct a geodesic segment which hits~$R$ in two points close to each other and both hits are as close to tangential as one wishes.
Continuing this segment to a broken ray shows that condition~\ref{cond:1tang-hit} is not satisfied.
\end{remark}

\begin{remark}
\label{rmk:convex2}
Condition~\ref{cond:1tang-hit} is not easy to check, but it is worth noting that it is between two more intuitive conditions:
\begin{itemize}
\item every broken ray hits~$R$ at most once (stronger than~\ref{cond:1tang-hit});
\item no broken ray has two tangential reflections (weaker than~\ref{cond:1tang-hit}).
\end{itemize}
Condition~\ref{cond:1tang-hit} demands the absence of two close to tangential reflections on a single broken ray in a uniform and strong way.
\end{remark}

The following result gives the basic example of a domain where the conditions in Theorem \ref{thm:brt} hold true:

\begin{proposition}
Let~$\ext{M}$ be an oriented, compact Riemann surface with nonpositive curvature and strictly convex boundary, and let $O\subset\ext{M}$ be a strictly convex smooth domain such that $\bar O\subset\sisus\ext{M}$.
Then $M=\ext{M}\setminus O$ satisfies the assumptions of Theorem~\ref{thm:brt} with $E=\partial\ext{M}$ and $R=\partial O$.
\end{proposition}

\begin{proof}
It suffices to show that every broken ray hits~$R$ at most once.
For contradiction, suppose we had a geodesic segment~$\gamma$ connecting two points on~$R$; a broken ray with two reflections gives rise to such a segment.

Let~$\sigma$ be the path joining the endpoints of~$\gamma$ along~$R$ such that~$\sigma$ and~$\gamma$ describe a domain~$\Omega$ such that $\bar\Omega\cap E=\emptyset$.
Parametrize $\sigma:[0,a]\to M$ and $\gamma:[0,b]\to M$ by arclength so that the concatenation $\sigma\wedge\gamma$ rotates counterclockwise around~$\Omega$.
Let~$\alpha$ be the oriented angle from~$\dot\gamma(b)$ to~$\dot\sigma(0)$ and~$\beta$ similarly from~$\dot\sigma(a)$ to~$\dot\gamma(0)$.
Let~$\nu$ denote the unit normal on~$R$ pointing out of~$M$.
Since~$\gamma$ is a geodesic, the Gauss-Bonnet theorem yields
\begin{equation}
\label{eq:GB}
\int_\Omega K\der\Sigma
-
\int_0^a\ip{D_t\dot\sigma}{\nu(\sigma(t))}\der t
=
2\pi-(\alpha+\beta),
\end{equation}
where~$\Sigma$ is the area measure of~$M$ and~$K$ is the Gaussian curvature.
By convexity of $O$ we have $\ip{D_t\dot\sigma}{\nu(\sigma(t))} > 0$ for all $t\in[0,a]$.
Also $K \leq 0$, so the left-hand side of~\eqref{eq:GB} is negative.
But~$\alpha$ and~$\beta$ cannot exceed~$\pi$ (the limit corresponds to~$\gamma$ hitting~$R$ tangentially), so the right-hand side is nonnegative.
This contradiction concludes the proof.
\end{proof}

\begin{remark}
\label{rmk:stability}
If one could prove the Pestov identity (Lemma~\ref{lma:pestov}) for~$u$ arising from~$f$ with nonvanishing broken ray transform, one obtains immediately the following stability result corresponding to Theorem~\ref{thm:brt}:
If the curvature of~$E$ is bounded from above by $\kappa_0>0$, we have the stability estimate
\begin{equation}
\label{eq:stab}
\frac{\kappa_0}{2\pi} \aabs{V u^f}_{L^2(SE)}\geq \aabs{f}_{L^2(M)}
\end{equation}
where~$V$ is the vertical vector field.
The function~$u^f$ is defined in~\eqref{eq:u-def}.
We emphasize that we have not proved~\eqref{eq:stab} since we have not proved our Pestov identity for sufficiently low regularity.
In the case of Euclidean plane a stability estimate was given by Eskin~\cite{E:hyperbolic}.
\end{remark}

An immediate corollary of Theorem \ref{thm:brt} in the Euclidean case is as follows:
If $\Omega\subset\R^2$ is a bounded smooth strictly convex domain and $O\subset\Omega$ is a smooth compact convex obstacle, then the broken ray transform is injective on $C^2(\bar{\Omega}\setminus O)$ with $E=\partial\Omega$ and $R=\partial O$. As discussed above, this result also follows immediately from the Helgason support theorem~\cite{book-helgason}. Other Helgason support theorems that would imply similar results are given in \cite{K:spt-thm} and \cite{UV:local-x-ray}. These support theorems assume that the manifold has a simple real-analytic metric or is at least three dimensional, but our result holds true also for negatively curved surfaces with smooth (not real-analytic) metric.
We do not rely on support theorems, since we aim towards a theory of the broken ray transform with any number of obstacles.
Our approach would allow to deal with any number of obstacles, provided that the regularity Lemma~\ref{lma:reg} could be proven in that case.

As illustrated in Figures~\ref{fig:ctr-ex1} and~\ref{fig:ctr-ex2}, there are planar domains with nonconvex obstacles or more than one convex obstacle such that the broken ray transform is not injective.
Therefore the assumptions of Theorem~\ref{thm:brt} have to include some restrictions on the geometry of~$R$.

\begin{figure}%
\includegraphics[width=\columnwidth/2]{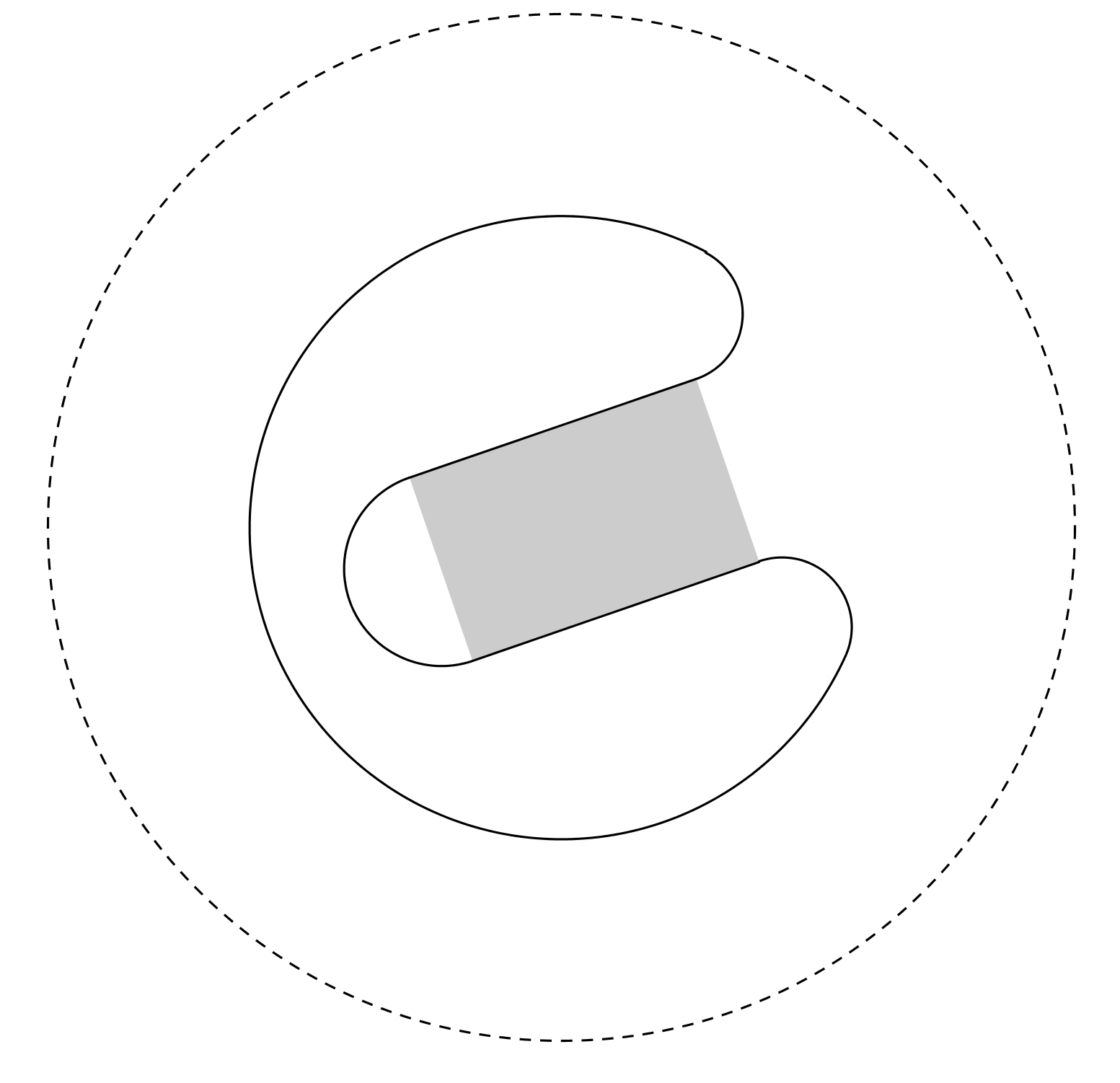}
\caption{A counterexample. The broken ray transform is not injective in this domain with an obstacle. Any ray passing through the shaded area has to go all the way through without changing velocity in the direction of the tube. Take any function in the gray tube depending only on the axial (not transversal) coordinate and integrating to zero and extend it by zero to the whole domain.}%
\label{fig:ctr-ex1}%
\end{figure}

\begin{figure}%
\includegraphics[width=\columnwidth/2]{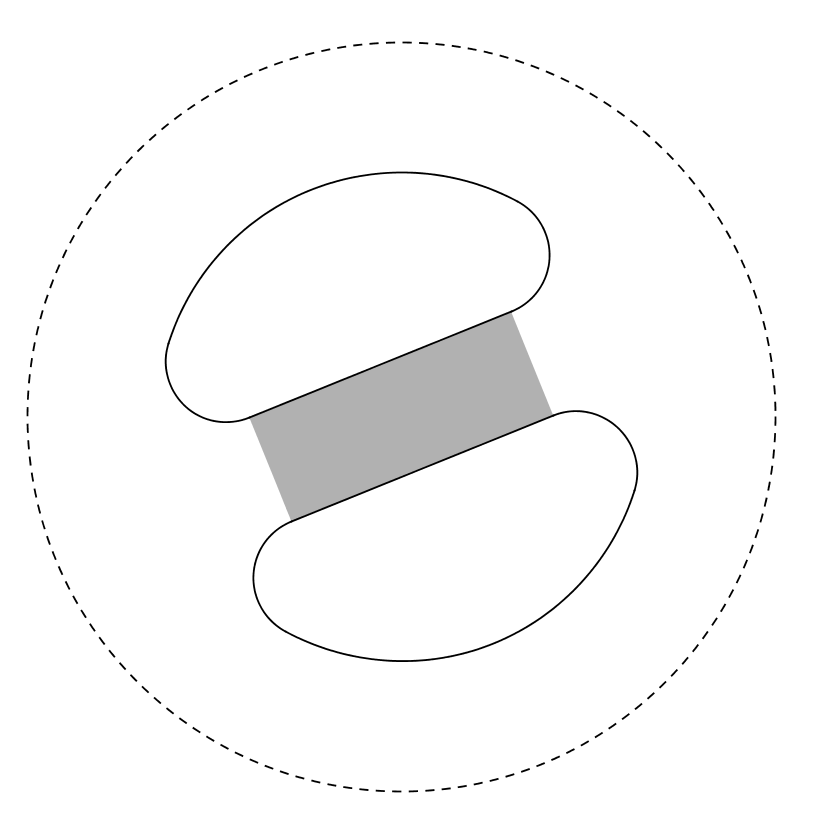}
\caption{A counterexample similar to that of Figure~\ref{fig:ctr-ex1} with two convex obstacles.}%
\label{fig:ctr-ex2}%
\end{figure}

As mentioned above, problems of this type have been studied by Mukhometov \cite{M1, M2, M3, M4, M5, M6}, and \cite{M4} states a theorem that is very close to our main result. Eskin~\cite{eskin} considered several reflecting obstacles in a Euclidean planar domain, but assumed the obstacles to have enough corners to eliminate periodic billiard trajectories.
For other recent results for the broken ray transform we refer to~\cite{I:disk,H:square,I:refl,I:bdy-det,H:brt-flat-refl}.
For results on the usual (non-reflected) geodesic ray transform, we refer to~\cite{S:tensor-book,PSU:tensor-surface,PSU:tensor-survey}.
The concept of broken rays also appears in different meanings, and another integral transform (the V-line Radon transform) is also known as the broken ray transform; see~\cite{KLU:broken-geodesic-flow,MNTZ:radon,TN:radon,A:radon}.

In Section~\ref{sec:preliminaries} we mention some facts related to analysis on the unit circle bundle required for our approach.
Section~\ref{sec:proof} gives an outline of the proof of Theorem~\ref{thm:brt}.
In Section~\ref{sec:pestov} we prove the Pestov identity with boundary terms in the form needed in this paper, and Section~\ref{sec:reg} establishes the required regularity for solutions of transport equations in order to apply the Pestov identity.
There we also discuss Jacobi fields related to reflected rays.
Finally, we will consider applications of the broken ray transform in Section~\ref{sec:appl}.

\subsection*{Acknowledgements}

J.I.\ and M.S.\ were partly supported by the Academy of Finland (Centre of Excellence in Inverse Problems Research), and M.S.\ was also supported by an ERC Starting Grant (grant agreement no 307023). The authors wish to express their gratitude to Vladimir Sharafutdinov for explaining the connection to the work of Mukhometov.

\section{Preliminaries}
\label{sec:preliminaries}

Here we discuss some preliminaries concerning analysis on the unit circle bundle. For more details we refer to~\cite{PSU:tensor-surface}. The Riemannian geometry notation in this paper mostly follows~\cite{L:riemann-book}.

Let $(M,g)$ be a compact oriented Riemann surface with smooth boundary.
The circle bundle~$SM$ is the set of unit tangent vectors on~$M$.
Let $\varphi_t$ be the geodesic flow on $SM$, so 
$$
\varphi_t(x,v) = (\gamma(t,x,v), \dot{\gamma}(t,x,v))
$$
where $\gamma(t,x,v)$ is the unit speed geodesic starting at $(x,v) \in SM$.

Let~$X$ be the geodesic vector field, defined for $w \in C^{\infty}(SM)$ by 
$$
Xw(x,v) = \frac{d}{dt} w(\varphi_t(x,v))\Big|_{t=0}.
$$
We wish to introduce two further vector fields on~$SM$, the vertical vector field~$V$ and the orthogonal vector field $X_{\perp}$. If $x = (x_1, x_2)$ is a system of oriented isothermal coordinates in $U \subset M$, so that the metric has the form $g_{jk} = e^{2\lambda} \delta_{jk}$ for some $\lambda \in C^{\infty}(U)$ in these coordinates, we may write~$\theta$ for the angle between~$v$ and~$\partial/\partial x_1$. Then~$SU$ may be identified with the set $\{ (x,\theta) \,;\, x \in U, \theta \in [0,2\pi) \}$, and the vertical vector field is given by 
$$
V = \partial_{\theta}.
$$
This defines invariantly a vector field~$V$ on~$SM$. The vector field~$X_{\perp}$ may be then defined as the commutator 
$$
X_{\perp} := [X,V].
$$
If~$SM$ is equipped with the natural metric induced by the Sasaki metric on~$TM$, it turns out that $\{X, X_{\perp}, V \}$ is a global orthonormal frame on~$SM$. We will also use the corresponding inner product on $L^2(SM)$ and the induced inner product on $L^2(\partial(SM))$.

We will need some very basic facts related to the calculus of semibasic tensor fields, see~\cite{S:tensor-book} and also~\cite{DKSU:anisotropic} for a version on~$SM$ (an alternative approach to this calculus is found in~\cite{PSU:anosov-mfld}). If $u \in C^{\infty}(SM)$, we write $u(x,y/\abs{y})$ for its extension as a function in $C^{\infty}(TM \setminus \{0\})$ that is homogeneous of degree zero. We then define the horizontal and vertical derivatives 
\begin{align*}
\nabla_{x_i} u &= \frac{\partial}{\partial x_i} (u(x,y/\abs{y})) - \Gamma_{ik}^l v^k \partial_{v_l} u\Big|_{SM}, \\
\partial_{v_i} u &= \frac{\partial}{\partial y_i} (u(x,y/\abs{y}))\Big|_{SM}.
\end{align*}
Here $\Gamma_{ik}^l$ are the Christoffel symbols of the metric~$g$. The quantities $\nabla u = \nabla_{x_i} u \,dx^i$ and $\partial u = \partial_{v_i} u \,dx^i$ are invariantly defined, and one has the identities 
$$
Xu = \langle v, \nabla u \rangle, \quad X_{\perp} u = -\langle v^{\perp}, \nabla u \rangle, \quad Vu = \langle v^{\perp}, \partial u \rangle.
$$
Here we identify $v$ with the corresponding $1$-form, and $v^{\perp}$ is the $g$-rotation of $v$ by $90$ degrees counterclockwise. We also observe that one always has $\langle v, \partial u \rangle = 0$.

\section{Outline of proof of Theorem \ref{thm:brt}}
\label{sec:proof}

Assume the conditions in Theorem \ref{thm:brt}. Given~$f\in C^2(M)$, suppose that the broken ray transform of $f$ vanishes. We define $u^f:SM\to\R$ by
\begin{equation}
\label{eq:u-def}
u^f(x,v)=\int_0^{\tau_{x,v}^E}f(\gamma_{x,v}^E(t))\der t.
\end{equation}
Since $u^f|_{E_+\cup E_0}=0$ and $u^f|_{E_-}$ is the broken ray transform of~$f$, we have~$u^f=0$ on~$SE$.

At~$SR$ we define the reflection map $\rho:SR\to SR$ by $\rho(x,v)=(x,v-2\langle v, \nu \rangle \nu)$.
Because of the reflection law imposed on broken rays, $u^f=u^f\circ\rho$ on~$SR$.
That is, the odd part of~$u^f$ with respect to~$\rho$ vanishes on~$SR$.

If $X$ is the geodesic vector field, it follows from the definition of~$u^f$ that
\begin{equation}
\label{eq:u-pde1}
Xu^f=-f
\end{equation}
in~$\sisus SM$. If~$V$ is the vertical vector field, applying it to the PDE~\eqref{eq:u-pde1} gives
\begin{equation}
\label{eq:u-pde2}
VXu^f=0
\end{equation}
in~$\sisus SM$.

Now we know that~$u^f$ satisfies the PDE~\eqref{eq:u-pde2} in $\sisus SM$ and it satisfies the boundary conditions $u^f|_{SE} = 0$, $u^f|_{SR} = u^f \circ \rho|_{SR}$.
From this information we would like to deduce that~$u^f = 0$ in all of~$SM$, hence showing that~$f=0$ which would prove injectivity of the broken ray transform.

The key identity is the following Pestov identity with boundary terms. Energy identities of this type were introduced by Mukhometov~\cite{M:bdy-rig-surface-eng,M4}, and in the case of $\R^2$ the estimate below was also proved by Eskin~\cite{eskin} . See \cite{S:tensor-book,PSU:tensor-surface,PSU:tensor-survey} for more facts about Pestov identities.

\begin{lemma}
\label{lma:pestov}
Let $(M,g)$ be a compact Riemann surface with smooth boundary that is contained in some disc $(D,g)$ in $\R^2$.
Suppose $u\in C(SM)\cap C^{1}(SM\setminus R_0)\cap C^2(SM\setminus SR)$.
If~$u$ satisfies $u=u\circ\rho$ on~$\partial SM$ and $VXu\in L^2(SM)$, then
\begin{equation}
\label{eq:pestov}
\begin{split}
\aabs{VX u}^2_{L^2(SM)}
&=
\aabs{XV u}^2_{L^2(SM)}
+
\aabs{X u}^2_{L^2(SM)}
\\&\quad
-
(KVu,Vu)_{L^2(SM)}
-
(\kappa Vu,Vu)_{L^2(\partial SM)},
\end{split}
\end{equation}
where~$K$ is the Gaussian curvature of~$M$ and~$\kappa$ is the signed curvature of~$\partial M$.
\end{lemma}

In order to use Lemma~\ref{lma:pestov}, we need to ensure that~$u^f$ has enough regularity.
This is provided by the following lemma:

\begin{lemma}
\label{lma:reg}
If the broken ray transform of~$f$ vanishes and the conditions~\ref{cond:E-convex}, \ref{cond:1tang-hit}, and~\ref{cond:obs} in Theorem~\ref{thm:brt} are satisfied, then $u^f$ is Lipschitz continuous in $SM$. Moreover, $u^f \in C^2(SM\setminus SR) \cap C^1(SM \setminus R_0)$.
\end{lemma}

With these results we can prove Theorem~\ref{thm:brt}.

\begin{proof}[Proof of Theorem~\ref{thm:brt}]
Suppose the broken ray transform of $f\in C^2(M)$ vanishes and define~$u^f$ by~\eqref{eq:u-def}.
Since $u^f|_{SE}=0$ and the vector field~$V$ is vertical, we have $Vu^f|_{SE}=0$.
Now $u^f|_{SR}$ is even with respect to~$\rho$ and Lemma~\ref{lma:reg} gives sufficient regularity, we may use Lemma~\ref{lma:pestov}.

Using the PDEs~\eqref{eq:u-pde1} and~\eqref{eq:u-pde2} the Pestov identity~\eqref{eq:pestov} becomes
\begin{equation}
\begin{split}
0
&=
\aabs{XV u^f}^2_{L^2(SM)}
+
\aabs{f}^2_{L^2(SM)}
\\&\quad
-
(KVu^f,Vu^f)_{L^2(SM)}
-
(\kappa Vu^f,Vu^f)_{L^2(SR)}.
\end{split}
\end{equation}
Since $K\leq0$ by assumption~\ref{cond:curv} and $\kappa\leq0$ by assumption~\ref{cond:1tang-hit} (see Remark~\ref{rmk:convex}), we have
\begin{equation}
0
\geq
\aabs{f}^2_{L^2(SM)}.
\end{equation}
This concludes the proof of injectivity.
\end{proof}

The estimate~\eqref{eq:stab} of Remark~\ref{rmk:stability} would follow similarly from the Pestov identity, but 
it is not obvious how to establish sufficient regularity for~$u^f$ when~$f$ has nonvanishing broken ray transform.
If one considers $f\equiv1$ and~$M$ is a Euclidean annulus, it becomes clear that $u^f$ is not in $C^2(SM)$; at least in the Euclidean case it should be in $C^{1/2}(SM)$.
Lemma~\ref{lma:JF-rfl-reg} and the proof of Lemma~\ref{lma:reg} suggest that $C^{1/2}(SM)$ regularity may be true on any~$M$ satisfying the assumptions of Theorem~\ref{thm:brt}.


\section{Proof of Lemma~\ref{lma:pestov}}
\label{sec:pestov}

We denote $P=VX$ and write $P^*=XV$ for the formal adjoint of~$P$.
The proofs of Lemmas~\ref{lma:pestov-smooth} and~\ref{lma:pestov-bdy} are given for $u\in C^\infty(SM)$; the claims for~$C^2$ follow by density of~$C^\infty$ in~$C^2$.

\begin{lemma}
\label{lma:pestov-smooth}
If $u \in C^2(SM)$ then
\begin{equation}
\begin{split}
\aabs{Pu}^2_{L^2(SM)}
&=
\aabs{P^* u}^2_{L^2(SM)} + \aabs{Xu}^2_{L^2(SM)}
\\&\quad
- (KVu,Vu)_{L^2(SM)} + (\nabla_T u, Vu)_{\partial(SM)}
\end{split}
\end{equation}
where $\nabla_T = \langle T, \nabla \rangle$ with $T = \nu^{\perp}$ the oriented tangent vector on~$\partial M$.
\end{lemma}

\begin{proof}
In this proof all norms and inner products will be in the space $L^2(SM)$ unless stated otherwise.

One has the integration by parts formulas for $w,z \in C^{\infty}(SM)$, 
\begin{gather*}
(Vw,z) = -(w,Vz), \\
(Xw,z) = -(w,Xz) + (\langle v,\nu \rangle w, z)_{\partial(SM)}.
\end{gather*}
Then 
\begin{align*}
&\aabs{Pu}^2 - \aabs{P^* u}^2 \\
 &\quad = (Pu, VXu) - (P^* u, XVu) \\
 &\quad = -(VPu,Xu) + (XP^* u,Vu) - (P^* u, \langle v,\nu \rangle Vu)_{\partial(SM)} \\
 &\quad = ([P^*,P]u, u) - (VPu, \langle v,\nu \rangle u)_{\partial(SM)} - (P^* u, \langle v,\nu \rangle Vu)_{\partial(SM)}.
\end{align*}
One has $[P^*,P] = -X^2 + VKV$ (see~\cite[Section 3]{PSU:tensor-surface}), which implies that 
$$
([P^*,P]u,u) = \aabs{Xu}^2 - (KVu,Vu) - (\langle v,\nu \rangle Xu,u)_{\partial(SM)}.
$$
Consequently  
\begin{multline*}
\aabs{Pu}^2 = \aabs{P^* u}^2 + \aabs{Xu}^2 - (KVu,Vu) \\
 - (VPu, \langle v,\nu \rangle u)_{\partial(SM)} - (P^* u, \langle v,\nu \rangle Vu)_{\partial(SM)} - (\langle v,\nu \rangle Xu,u)_{\partial(SM)}.
\end{multline*}
To simplify the boundary term we note that for $w, z \in C^{\infty}(S_x)$ 
$$
\int_{S_x} (Vw) {z} \,\der S_x = - \int_{S_x} w {Vz} \,\der S_x.
$$
Using that $P - P^* = [V,X] = -X_{\perp}$, the boundary term becomes 
$$
-(X_{\perp} u, \langle v,\nu \rangle Vu)_{\partial(SM)} + (Pu, V(\langle v,\nu \rangle) u)_{\partial(SM)} - (\langle v,\nu \rangle Xu,u)_{\partial(SM)}.
$$
Since $P = VX$, integrating by parts once more with respect to~$V$ shows that the boundary term will be 
\begin{multline*}
-(X_{\perp} u, \langle v,\nu \rangle Vu)_{\partial(SM)} - (Xu, V(\langle v,\nu \rangle) Vu)_{\partial(SM)} \\
 - (Xu, V^2(\langle v,\nu \rangle) u)_{\partial(SM)} - (\langle v,\nu \rangle Xu,u)_{\partial(SM)}.
\end{multline*}
One has $V(\langle v,\nu \rangle) = \langle v^{\perp},\nu \rangle$ and $V^2(\langle v,\nu \rangle) = -\langle v,\nu \rangle$. This shows that the last two terms in the boundary term will cancel, and the boundary terms simplify to 
\begin{align*}
&(\langle v,\nu \rangle \langle v^{\perp}, \nabla u \rangle,Vu)_{\partial(SM)} - (\langle v^{\perp},\nu \rangle \langle v,\nabla u \rangle, Vu)_{\partial(SM)} \\
 &\quad= (\langle v,\nu^{\perp} \rangle \langle v,\nabla u \rangle, Vu)_{\partial(SM)} + (\langle v^{\perp},\nu^{\perp} \rangle \langle v^{\perp}, \nabla u \rangle, Vu)_{\partial(SM)} \\
 &\quad= (\langle \nu^{\perp}, \nabla u \rangle, Vu)_{\partial(SM)}.\qedhere
\end{align*}
\end{proof}

\begin{lemma}
\label{lma:pestov-bdy}
If $u \in C^2(SM)$ then
\begin{equation}
\begin{split}
(\nabla_T u, Vu)_{\partial(SM)} &= (\nabla_T u_e, Vu_o)_{\partial(SM)} + (\nabla_T u_o, Vu_e)_{\partial(SM)} \\
 &\quad \quad - (\kappa Vu, Vu)_{\partial(SM)}
\end{split}
\end{equation}
where $\kappa := -\langle D_t T, \nu \rangle$ is the signed curvature of~$\partial M$, and~$u_e$ and~$u_o$ are the even and odd components of~$u|_{\partial(SM)}$ with respect to the reflection~$\rho$.
\end{lemma}

\begin{proof}
Let $x = (x^1,x^2)$ be positively oriented isothermal coordinates on $(M,g)$, so that the metric has the form 
$$
g_{jk} = e^{2\lambda} \delta_{jk}
$$
for some smooth function~$\lambda$. Let $(x,\theta)$ be corresponding coordinates on~$SM$, defined via 
$$
(x,\theta) \mapsto (x,v), \quad v(x,\theta) := e^{-\lambda(x)} \omega_{\theta}
$$
where $\omega_{\theta} := (\cos \theta, \sin \theta)$. We may assume that we are working on a fixed component of~$\partial M$, given by the oriented unit speed curve $\gamma(t)$. Let $\alpha(\gamma(t))$ be the angle of $\dot{\gamma}(t)$:
$$
\dot{\gamma}(t) = e^{-\lambda(\gamma(t))} \omega_{\alpha(\gamma(t))}.
$$
Below we will write $\alpha(t)$ instead of $\alpha(\gamma(t))$. It follows that in the $(x,\theta)$ coordinates, the reflection is given by 
$$
\rho(\gamma(t),\theta) = (\gamma(t),2\alpha(t) - \theta).
$$

The vertical vector field is given by 
$$
V = \partial_{\theta}.
$$
We wish to compute~$\nabla_T$, which involves the horizontal derivative 
$$
\nabla_j u = \partial_{x_j} \tilde{u} - \Gamma_{jk}^l v^k \partial_{v_l} u
$$
where $\tilde{u}$ is the homogeneous degree $0$ extension of $u \in C^{\infty}(SM)$ to $TM \setminus \{ 0 \}$. Using the form of the metric, the Christoffel symbols are 
$$
\Gamma_{jk}^l = (\partial_j \lambda) \delta_k^l + (\partial_k \lambda) \delta_j^l - (\partial_l \lambda) \delta_{jk}.
$$
The vertical gradient $\partial u := \partial_{v_l} u \,dx^l$ satisfies 
$$
\langle v, \partial u \rangle = 0, \quad \langle v^{\perp}, \partial u \rangle = \partial_{\theta} u.
$$
Consequently 
$$
\nabla_j u = \partial_{x_j} \tilde{u} - \langle d\lambda, v \rangle \partial_{v_j} u + \langle d\lambda, \partial u \rangle v_j
$$
and 
$$
\nabla_T u = T^j \nabla_j u = T^j \partial_{x_j} \tilde{u} - [ \langle d\lambda, v \rangle \langle T, v^{\perp} \rangle - \langle d\lambda, v^{\perp} \rangle \langle T, v \rangle ] \partial_{\theta} u.
$$
The expression in brackets is a determinant and independent of~$v$:
\begin{gather*}
\nabla_T u(x,\theta) = T^j(x) \partial_{x_j} \tilde{u}(x,\theta) + c(x) \partial_{\theta} u(x,\theta), \ \ x \in \partial M, \\
c(x) := \det(T(x), d\lambda(x)) = T^1 \partial_2 \lambda - T^2 \partial_1 \lambda.
\end{gather*}

Let us determine how~$\nabla_T$ and~$V$ behave under reflection. One has 
\begin{align*}
(\rho^* Vu)(\gamma(t),\theta) &= \partial_{\theta} u(\gamma(t),2\alpha(t) - \theta) \\
 &= -V(\rho^* u)(\gamma(t),\theta)
\end{align*}
and, after a short computation,  
\begin{align*}
 &(\rho^* \nabla_T u)(\gamma(t),\theta) = (\nabla_T u)(\gamma(t),2\alpha(t)-\theta) \\
 &= \dot{\gamma}^j(t) \partial_{x_j} \tilde{u}(\gamma(t),2\alpha(t)-\theta) + c(\gamma(t)) \partial_{\theta} u(\gamma(t),2\alpha(t)-\theta) \\
 &= \partial_t (u(\gamma(t),2\alpha(t)-\theta)) - 2 \dot{\alpha}(t) \partial_{\theta} u(\gamma(t),2\alpha(t)-\theta) \\
 &\quad \quad + c(\gamma(t)) \partial_{\theta} u(\gamma(t),2\alpha(t)-\theta) \\
 &= \nabla_T (\rho^* u)(\gamma(t),\theta) + 2[c(\gamma(t)) - \dot{\alpha}(t)] (\rho^* Vu)(\gamma(t),\theta).
\end{align*}
We may write the part in brackets in terms of the signed curvature of~$\partial M$, since 
$$
\dot{\alpha}(t) = \kappa(\gamma(t)) + \eta(\dot{\gamma}(t))
$$
where $\eta = \partial_2 \lambda \,dx^1 - \partial_1 \lambda \,dx^2$ (see \cite[Theorem 9.3]{L:riemann-book}). Thus $\eta(\dot{\gamma}(t)) = c(\gamma(t))$ and 
\begin{align*}
(\rho^* \nabla_T u)(\gamma(t),\theta)  = \nabla_T (\rho^* u)(\gamma(t),\theta) - 2 \kappa(\gamma(t)) (\rho^* Vu)(\gamma(t),\theta).
\end{align*}
These facts imply that 
\begin{gather*}
(\nabla_T u)_e = \nabla_T u_e - \kappa \rho^* Vu, \quad (\nabla_T u)_o = \nabla_T u_o + \kappa \rho^* Vu, \\
(Vu)_e = V u_o, \quad (Vu)_o = V u_e.
\end{gather*}

Since~$\rho$ is an isometry on~$S_x$ for each $x \in \partial M$, the $L^2(\partial(SM))$ inner product of an even function and an odd one vanishes and we obtain 
\begin{align*}
(\nabla_T u, Vu)_{\partial(SM)} &= ((\nabla_T u)_e, (Vu)_e)_{\partial(SM)} + ((\nabla_T u)_o, (Vu)_o)_{\partial(SM)} \\
 &= (\nabla_T u_e, Vu_o)_{\partial(SM)} - (\kappa Vu_o, Vu_o)_{\partial(SM)} \\
 &\quad + (\nabla_T u_o, Vu_e)_{\partial(SM)} - (\kappa V u_e, V u_e)_{\partial(SM)}.
\end{align*}
This proves the result.
\end{proof}

Now we are ready to begin the proof of Lemma~\ref{lma:pestov}.

\begin{proof}[Proof of Lemma~\ref{lma:pestov}]
For $\eps>0$ denote $M_\eps=\{x\in M;d(x,R)\geq\eps\}$ and $R_\eps=\{x\in M;d(x,R)=\eps\}$.
If~$\eps$ is small enough, $M_\eps$ is a Riemannian surface with boundary, with two boundary components~$E$ and~$R_\eps$.

We can extend the normal vector field~$\nu$ from~$R$ to $M\setminus M_\eps$ for~$\eps$ small enough so that~$\nu$ is normal to~$R_\eps$.
Let $SR_\eps^\delta=\{(x,v)\in SM; x\in R_\eps, \abs{\ip{v}{\nu}}<\delta\}$ and $SR^\delta\subset SR$ similarly.

By the assumption $u\in C^2(SM\setminus SR)$ we have that $u\in C^2(SM_\eps)$.
Using Lemmas~\ref{lma:pestov-smooth} and~\ref{lma:pestov-bdy} for~$M_\eps$ instead of~$M$ we obtain
\begin{equation}
\label{eq:vv1}
\begin{split}
\aabs{Pu}_{L^2(SM_\eps)}^2
&=
\aabs{P^* u}_{L^2(SM_\eps)}^2
+ \aabs{Xu}_{L^2(SM_\eps)}^2
\\&\quad
- (KVu,Vu)_{L^2(SM_\eps)}
+ (\nabla_T u_e, Vu_o)_{\partial(SM_\eps)}
\\&\quad
+ (\nabla_T u_o, Vu_e)_{\partial(SM_\eps)}
- (\kappa Vu, Vu)_{\partial(SM_\eps)}.
\end{split}
\end{equation}
We partition $\partial(SM_\eps)$ as $SR^\delta_\eps\cup(\partial(SM_\eps)\setminus SR^\delta_\eps)$.

We then let $\eps\to0$ with $\delta>0$ fixed.
Since $u\in C^1(SM\setminus SR^\delta)$, all of the inner products in~\eqref{eq:vv1} satisfy $(\,\cdot\,,\,\cdot\,)_{\partial(SM_\eps)\setminus SR^\delta_\eps}\to (\,\cdot\,,\,\cdot\,)_{\partial(SM)\setminus SR^\delta}$ as $\eps\to0$.
But since~$u$ is also Lipschitz, its first order derivatives are uniformly bounded in $SM$.
Thus the inner products evaluated on $SR_\eps^\delta$ and $SR^\delta$ are bounded by a constant (depending on~$u$) times~$\delta$.
The parameter~$\delta$ can be chosen arbitrarily small, so each inner product in~\eqref{eq:vv1} satisfies $(\,\cdot\,,\,\cdot\,)_{\partial(SM_\eps)}\to (\,\cdot\,,\,\cdot\,)_{\partial(SM)}$ as $\eps\to0$.

Also,~$u$ being Lipschitz implies that $Xu\in L^2(SM)$ and we have assumed that $Pu\in L^2(SM)$, so their norms in $L^2(SM_\eps)$ converge to the norms in $L^2(SM)$ as $\eps\to0$.
Since~\eqref{eq:vv1} holds for all~$\eps$ and all terms but one are known to converge, also $\aabs{P^* u}_{L^2(SM_\eps)}^2$ must have a finite, nonnegative limit, which implies that $P^*u\in L^2(SM)$.


We have found that
\begin{equation}
\begin{split}
\aabs{Pu}_{L^2(SM)}^2
&=
\aabs{P^* u}_{L^2(SM)}^2
+ \aabs{Xu}_{L^2(SM)}^2
\\&\quad
- (KVu,Vu)_{L^2(SM)}
+ (\nabla_T u_e, Vu_o)_{\partial(SM)}
\\&\quad
+ (\nabla_T u_o, Vu_e)_{\partial(SM)}
- (\kappa Vu, Vu)_{\partial(SM)}.
\end{split}
\end{equation}
But $u$ is even by assumption, so $u_o=0$ on~$SR$, and also $u=0$ on $SE$.
This concludes the proof of Lemma~\ref{lma:pestov}.
\end{proof}

\section{Proof of Lemma~\ref{lma:reg}}
\label{sec:reg}

We begin by some preparations, and then prove the claims of Lemma~\ref{lma:reg} in Sections~\ref{sec:ddu} and~\ref{sec:du}.
The preparations aim at understanding how a broken ray depends on its initial direction; this dependence is captured in the concept of \emph{Jacobi fields along broken rays}.

\subsection{Jacobi fields along broken rays}
\label{sec:JF-BR}

We begin by a basic observation about Jacobi fields along geodesics.

\begin{lemma}
\label{lma:jacobi}
On a smooth compact Riemannian manifold any Jacobi field~$J$ along a unit speed geodesic satisfies
\begin{equation}
\label{eq:est-jacobi}
\abs{J(t)}^2+\abs{D_tJ(t)}^2\leq e^{Ct}(\abs{J(0)}^2+\abs{D_tJ(0)}^2)
\end{equation}
for all $t\geq0$, where~$C$ is a uniform constant.
\end{lemma}

\begin{proof}
Let~$\gamma$ be the geodesic in question.
Let $Z_i$, $i=1,2$, be two vector fields along~$\gamma$.
Define $Z=(Z_1,Z_2)$ and suppose it satisfies
\begin{equation}
\label{eq:de-jacobi}
\begin{split}
D_tZ_1&=Z_2\quad\text{and}\\
D_tZ_2&=-R(Z_1,\dot{\gamma})\dot{\gamma},
\end{split}
\end{equation}
where~$R$ is the Riemann curvature tensor.
Then $D_tZ=A^\gamma_tZ$, where the mapping $A^\gamma_t$ is linear for each~$t$.
By compactness there is a constant~$C$ such that $\aabs{A^\gamma_t}\leq C/2$ for all geodesics~$\gamma$ and all times~$t$.

Thus
\begin{equation}
D_t\abs{Z}^2
=
2\ip{Z}{A^\gamma_tZ}
\leq
C\abs{Z}^2.
\end{equation}
By Gr\"onwall's inequality this implies that $\abs{Z(t)}^2\leq e^{Ct}\abs{Z(0)}^2$ for all $t\geq0$.

If $J$ is a Jacobi field, then $Z=(J,D_tJ)$ satisfies the equations~\eqref{eq:de-jacobi}, from which the claim follows.
\end{proof}

We are now ready to introduce Jacobi fields along a broken ray.
Jacobi fields along geodesics can be understood as infinitesimal geodesic deviations, and we want to generalize this idea to broken rays.
The key problem is to find the correct behaviour of Jacobi fields at reflection points.

We begin by introducing some notation.
Let $x_0\in\partial M$ be a point and~$\nu$ the outward unit normal at it.
We recall that that the reflection map $\rho:T_{x_0}M\to T_{x_0}M$ is defined by
\begin{equation}
\rho \xi
=
\xi - 2\ip{\xi}{\nu}\nu.
\end{equation}
For any vector $\zeta\in T_{x_0}M$ that is not orthogonal to~$\nu$ we define the map $\Phi_\zeta:T_{x_0}M\to T_{x_0}M$ by
\begin{equation}
\Phi_\zeta \xi
=
2(
\ip{\nabla_{\phi_\zeta\xi}\nu}{\zeta}\nu
+
\ip{\nu}{\zeta}\nabla_{\phi_\zeta\xi}\nu
).
\end{equation}
Here the map $\phi_\zeta: T_{x_0}M\to T_{x_0}M$ is
\begin{equation}
\phi_\zeta\xi
=
\xi - \frac{\ip{\xi}{\nu}}{\ip{\zeta}{\nu}}\zeta,
\end{equation}
and it is easy to see that always $\phi_\zeta\xi\perp\nu$.
Since $\phi_\zeta\xi\perp\nu$ and~$\nu$ is a vector field defined on~$\partial M$, the derivative $\nabla_{\phi_\zeta\xi}\nu$ is well defined. One also has $\nabla_{\phi_\zeta\xi}\nu = s(\phi_\zeta\xi)$ where $s$ is the shape operator of $\partial M \subset M$, that is, the map $s: T(\partial M) \to T(\partial M)$ defined by $s(X) = \nabla_X \nu$.

The maps~$\rho$, $\phi_\zeta$ and~$\Phi_\zeta$ are linear and have the following properties:
\begin{equation}
\begin{split}
\rho\circ\rho&=\id,\\
\phi_\zeta\circ\phi_\zeta&=\phi_\zeta,\\
\rho\circ\phi_\zeta&=\phi_\zeta,\\
\phi_{-\zeta}&=\phi_\zeta,\\
\phi_{\rho\zeta}\circ\rho&=\phi_\zeta,\\
\phi_\zeta\zeta&=0,\\
\Phi_{-\zeta}&=-\Phi_\zeta,\\
\Phi_{\rho\zeta}\circ\rho&=-\rho\circ\Phi_\zeta
.
\end{split}
\end{equation}
These properties are easy to check when one keeps in mind that $\nabla_{\phi_\zeta\xi}\nu$ is orthogonal to~$\nu$ and the map~$\rho$ only changes the sign of the normal component of a vector.

\begin{definition}
Let~$\gamma$ be a broken ray without tangential reflections.
Then a vector field~$J$ along~$\gamma$ is a \emph{Jacobi field along}~$\gamma$ if
\begin{itemize}
\item
it is a Jacobi field along the geodesic segments of~$\gamma$ in the usual sense
and
\item
if~$\gamma$ has a reflection at $\gamma(t_0)\in\partial M$, then the left and right limits of~$J$ at~$t_0$ are related via
\begin{equation}
\label{eq:JF-rfl}
J(t_0-)=\rho J(t_0+)
\text{ and }
D_tJ(t_0-)=\rho D_tJ(t_0+)-\Phi_{\dot{\gamma}(t_0+)}J(t_0+)
\end{equation}
or equivalently
\begin{equation}
J(t_0+)=\rho J(t_0-)
\text{ and }
D_tJ(t_0+)=\rho D_tJ(t_0-)-\Phi_{\dot{\gamma}(t_0-)}J(t_0-)
.
\end{equation}
\end{itemize}
\end{definition}

The next lemma shows that this is the definition of a Jacobi field along a broken ray we sought for.

\begin{lemma}
\label{lma:JF-def}
Let $(-\eps,\eps)\ni s\mapsto (x_s,v_s)\in \sisus SM$ be a~$C^1$ map and denote by~$\gamma_s: [0,T] \to M$ the broken ray starting at $(x_s,v_s)$.
Suppose none of the broken rays~$\gamma_s$ have tangential reflections up to time~$T$.
Then
\begin{equation}
J(t)
=
\left.\Der{s}\gamma_s(t)\right|_{s=0}
\end{equation}
is a Jacobi field along the broken ray~$\gamma_0$.

Conversely, any Jacobi field can be realised in this way as a variation of the broken ray~$\gamma_0$.
\end{lemma}

\begin{proof}
It is obvious that~$J$ satisfies the Jacobi equation on geodesic segments, so the only matter to check is behaviour at reflections.
We assume that $\gamma\coloneqq\gamma_0$ is a broken ray with a non-tangential reflection at $t=t_0$ and (after possibly shrinking the domain of definition) this is the only reflection of~$\gamma$ and moreover each $\gamma_s$ has only one reflection.
We will show that~$J$ satisfies the condition~\eqref{eq:JF-rfl}.

Let us start with the easiest Jacobi field.
Namely, fix $(x_0,v_0)$ and let $x_s=\gamma(s)$ and $v_s=\dot\gamma(s)$, which leads to $\gamma_s(t)=\gamma(t+s)$.
The corresponding vector field (denoted by~$I$ instead of~$J$) is $I(t)=\dot\gamma(t)$.
By definition of a broken ray we have $I(t_0-)=\rho I(t_0+)$.
Now we also have $D_tI(t)=0$ outside reflections, so $D_tI(t_0-)=0=D_tI(t_0+)$.
By the property $\phi_\zeta\zeta=0$ the field~$I$ satisfies~\eqref{eq:JF-rfl}.

Let then~$J$ be any vector field along~$\gamma$ as in the statement of the lemma.
We define another vector field~$\tilde J$ along~$\gamma$ by
\begin{equation}
\tilde J(t)
=
J(t) - \frac{\ip{J(t_0-)}{\nu}}{\ip{I(t_0-)}{\nu}}I(t).
\end{equation}
It suffices to show that~$\tilde J$ satisfies~\eqref{eq:JF-rfl}; this is because~$I$ satisfies it and the condition is linear.
The important property of~$\tilde J$ is that~$\tilde J(t_0-)\perp\nu$.
The vector field~$\tilde J$ can also be realized by suitable initial directions $(\tilde x_s,\tilde v_s)$ as in the statement of the lemma.

It thus suffices to show that~$J$ satisfies~\eqref{eq:JF-rfl} under the additional assumption that $J(t_0-)\perp\nu$.
Since~$\rho$ is identity on tangential vectors, we only need to verify the second part of~\eqref{eq:JF-rfl}.
We observe that~$J$ only depends on $\left.\Der{s}(x_s,v_s)\right|_{s=0}$, so we may make changes of order~$s^2$ to $(x_s,v_s)$ without altering~$J$.
Because $\dot\gamma(t_0-)$ is not tangential and $J(t_0-)\perp\nu$, we can make such a second order change to $(x_s,v_s)$ that we have $\gamma_s(t_0)\in\partial M$ for all~$s$ (this is possible since the curve $s \mapsto \gamma_s(t_0)$ is tangent to $\partial M$ at $s=0$).

We now shift time so that $t_0=0$.
We have a family~$\gamma_s$ of broken rays defined near $t=0$ with their only reflection at $t=0$.
We write $y_s=\gamma_s(0)$ and $u_s=\dot\gamma_s(0+)$.
The normal vector at~$y_s$ is denoted by~$\nu_s$ and the corresponding reflection map by~$\rho_s$.

Now $D_tJ(0+)=\left.D_su_s\right|_{s=0}$ and $D_tJ(0-)=\left.D_s\rho_su_s\right|_{s=0}$.
Additionally $J(0\pm)=\left.\Der{s}y_s\right|_{s=0}$.
Thus evaluating the identity
\begin{equation}
\begin{split}
D_s(\rho_su_s)
&=
D_s(u_s-2\ip{u_s}{\nu_s}\nu_s)
\\&=
D_su_s-2\ip{D_su_s}{\nu_s}\nu_s
\\&\quad-2(\ip{u_s}{D_s\nu_s}\nu_s+\ip{u_s}{\nu_s}D_s\nu_s)
\\&=
\rho_s(D_su_s)
-2(\ip{u_s}{\nabla_{\partial_sy_s}\nu_s}\nu_s+\ip{u_s}{\nu_s}\nabla_{\partial_sy_s}\nu_s)
\end{split}
\end{equation}
at $s=0$ gives
\begin{equation}
D_tJ(0-)
=
\rho_0D_tJ(0+)
-
\Phi_{\dot\gamma(0+)}J(0+).
\end{equation}
This identity finally concludes the proof of the first claim.

For the converse, let~$J$ be a Jacobi field.
Then before the first reflection it is a Jacobi field in the usual geodesic sense and can be realised as a geodesic variation.
If we continue these varied geodesics to broken rays, the resulting variation field is precisely~$J$, because the above calculation shows that Jacobi fields and variations of broken rays satisfy the same condition at reflection points.
\end{proof}

\begin{remark}
The assumption that the reflections are nontangential is important in the definition of the Jacobi field.
If there is a tangential reflection, the varied broken ray~$\gamma_s$ does not generally have~$C^1$ dependence on~$s$ and thus~$J$ is not well defined after the reflection.

We know how exactly a Jacobi field blows up when a reflection becomes more and more tangential.
The reflection condition is
\begin{equation}
\label{eq:JF-blowup}
\begin{split}
J(t_0-)&=\rho J(t_0+)
\text{ and}\\
D_tJ(t_0-)&=\rho D_t J(t_0+)+\ip{\dot\gamma(t_0+)}{\nu}^{-1}A J(t_0+),
\end{split}
\end{equation}
where~$A$ is a linear map satisfying uniform bounds on $SR$.
The map~$A$ encodes the curvature of~$\partial M$ at the reflection point.
Equation~\eqref{eq:JF-blowup} gives a good description of the nature of the singularity.
\end{remark}

We formulate the observation of the above remark as a lemma:

\begin{lemma}
\label{lma:JF-rfl-reg}
On a compact smooth manifold~$M$ with boundary a Jacobi field~$J$ along a broken ray~$\gamma$ satisfies
\begin{equation}
\abs{J(t_0+)}^2+\abs{D_tJ(t_0+)}^2
\leq
\frac{C}{\ip{\dot\gamma(t_0-)}{\nu}}
\left(
\abs{J(t_0-)}^2+\abs{D_tJ(t_0-)}^2
\right)
\end{equation}
at every reflection point~$t_0$, where~$C$ is a constant depending on~$M$.
\end{lemma}

\begin{corollary}
\label{cor:JF-growth}
Let~$M$ be a smooth Riemannian manifold with boundary.
Fix a number $a\in(0,1]$.
Consider those broken rays~$\gamma$ on~$M$ for which $\abs{\ip{\dot\gamma}{\nu}}\geq a$ at every reflection point.
Then any Jacobi field~$J$ along such a broken ray satisfies
\begin{equation}
\abs{J(t)}^2+\abs{D_tJ(t)}^2
\leq
Ae^{Bt}
\left(
\abs{J(0)}^2+\abs{D_tJ(0)}^2
\right)
\end{equation}
for all $t\geq0$, where~$A$ and~$B$ are constants depending on~$M$ and~$a$.
\end{corollary}

\begin{proof}
Let us call the broken rays satisfying the given condition admissible broken rays.
By compactness and the transversality condition $\abs{\ip{\dot\gamma}{\nu}}\geq a$ there is a number~$L$ such that the distance between consecutive reflection points of any admissible broken ray is at least~$L$.
For brevity, we write $\aabs{J}^2\coloneqq\abs{J}^2+\abs{D_tJ}^2$.

By the transversality condition and Lemma~\ref{lma:JF-rfl-reg} we know that there is a constant~$A$ such that~$\aabs{J}^2$ can only increase by a factor~$A$ at a reflection point of an admissible broken ray.
On the other hand, Lemma~\ref{lma:jacobi} bounds the growth of Jacobi fields on geodesic segments; $\aabs{J(t)}^2\leq e^{Ct}\aabs{J(0)}^2$ for positive~$t$.

Let~$N(t)$ be the number of reflections~$\gamma$ has in the time interval $(0,t)$.
This number has the estimate $N(t)\leq 1+t/L$.
Combining our findings, we get for positive times~$t$
\begin{equation}
\begin{split}
\aabs{J(t)}^2
&\leq
A^{N(t)}e^{Ct}\aabs{J(0)}^2
\\&\leq
Ae^{Bt}\aabs{J(0)}^2,
\end{split}
\end{equation}
where $B=C+\log(A)/L$.
\end{proof}

The corollary can also be formulated differently, but the proof is the same:

\begin{corollary}
Let~$M$ be a smooth Riemannian manifold with boundary.
Let~$\gamma$ be any broken ray on~$M$ without tangential reflections.
Then any Jacobi field~$J$ along~$\gamma$ satisfies
\begin{equation}
\abs{J(t)}^2+\abs{D_tJ(t)}^2
\leq
\left(\prod_{\substack{\text{\normalfont reflections}\\\text{\normalfont before }t}}\frac{A}{\abs{\ip{\dot\gamma}{\nu}}}\right)e^{Ct}
\left(
\abs{J(0)}^2+\abs{D_tJ(0)}^2
\right)
\end{equation}
for all $t\geq0$, where~$A$ and~$C$ are constants depending on~$M$.
\end{corollary}

Now that we have found the natural definition for Jacobi fields along a broken ray, we could easily define conjugate points along a broken ray. We do not pursue this direction here.

\subsection{Existence of second derivatives}
\label{sec:ddu}

Let $u^f$ be as in \eqref{eq:u-def}. For all $(x,v)\in\sisus SM$ we have
\begin{equation}
\label{eq:u-symm}
u^f(x,v)+u^f(x,-v)=0,
\end{equation}
because the sum is an integral of $f\in C^2(M)$ over a broken ray passing through $(x,v)$.
Because of assumption~\ref{cond:1tang-hit} we know that at least one of $\gamma_{x,\pm v}^E$ has no reflections with $\abs{\ip{\dot\gamma}{\nu}}<a$.

Therefore it suffices to prove existence of second order derivatives of $u^f$ in the absence of reflections with $\abs{\ip{\dot\gamma}{\nu}}<a$. 
But since $f\in C^2(M)$ and~$M$ is compact, this follows immediately from~\eqref{eq:u-def} and smooth dependence of a broken ray and its endpoint on the initial point and direction (in the absence of tangential reflections).

\subsection{Boundedness of first derivatives}
\label{sec:du}


As for the existence of second derivatives it suffices to prove boundedness of first order derivatives of~$u^f$ in the absence of reflections with $\abs{\ip{\dot\gamma}{\nu}}\leq a/2$.
Suppose $(x,v)\in SM$ is such that $\gamma_{x,v}^E$ has no such reflections, and let $(-\eps,\eps)\ni s\mapsto(x_s,v_s)$ be a unit speed~$C^1$ curve on~$SM$.
Taking~$\eps$ small enough we can assume that $\gamma_{x_s,v_s}^E$ has no such reflections for any~$s$.
We need to show that $s\mapsto u^f(x_s,v_s)$ has a derivative at $s=0$ and it has an upper bound independent of the choice of~$(x,v)$.

We have
\begin{equation}
\label{eq:vv2}
\Der{s}u^f(x_s,v_s)
=
f(\gamma_{x_s,v_s}^E(\tau_{x_s,v_s}^E))\Der{s}\tau_{x_s,v_s}^E
+
\int_0^{\tau_{x_s,v_s}^E}\Der{s}f(\gamma_{x_s,v_s}^E(t))\der t.
\end{equation}
For the first term we use boundary determination: since $E$ is strictly convex, for any $x \in E$ we can choose a tangential vector $(x,v_0)$ and a sequence $(x,v_k) \in E_-$ with $(x,v_k) \to (x,v_0)$. Since the geodesics in direction $(x,v_k)$ become arbitrarily short and $f$ is continuous, we have 
\begin{align*}
f(x) &= \lim_{k \to \infty} \frac{1}{\tau^E_{x,v_k}} \int_0^{\tau^E_{x,v_k}} f(\gamma_{x,v_k}(t)) \,\der t \\
 &= \lim_{k \to \infty} \frac{1}{\tau^E_{x,v_k}}u^f(x,v_k) = 0.
\end{align*}
Thus the first term on the right hand side of~\eqref{eq:vv2} vanishes (notice that~$\tau^E$ is smooth near $(x,v)$ whenever the broken ray~$\gamma^E_{x,v}$ reflects and exits transversally).  

For the second term in \eqref{eq:vv2}, we use that $J_s(t)=\Der{s}\gamma_{x_s,v_s}^E(t)$ is a Jacobi field along the broken ray $\gamma_{x_s,v_s}^E$ and
\begin{equation}
\abs{J_s(0)}^2+\abs{D_tJ_s(0)}^2
=
\abs{\Der{s}x_s}^2+\abs{D_sv_s}^2
=
1.
\end{equation}
Thus by Corollary~\ref{cor:JF-growth} and assumption~\ref{cond:obs} $\abs{J_s(t)}$ has some uniform bound.
The integrand in~\eqref{eq:vv2} is
\begin{equation}
\Der{s}f(\gamma_{x_s,v_s}^E(t))
=
\nabla_Jf(\gamma_{x_s,v_s}^E(t)).
\end{equation}
Since~$\abs{\nabla f}$ is bounded and the length of the interval of integration is bounded by assumption~\ref{cond:obs}, the second term in~\eqref{eq:vv2} is bounded as well. 

It now remains to extend these estimates to~$SR$.
The same arguments hold true when $x\in R$ provided that $v\in S_xM$ is not tangential.
Since broken rays depend continuously on their initial point and direction (even across tangential reflections), we see that $u^f\in C(SM)$.
Thus~$u^f$ is bounded.

\section{Applications of the broken ray transform}
\label{sec:appl}

We conclude with describing some applications of the broken ray transform.
Section~\ref{sec:lens} below is devoted to connecting the broken ray transform with lens data, but we will first describe some other applications briefly.

The original motivating application for the broken ray transform is in inverse boundary value problems for partial differential equations.
Eskin~\cite{eskin} studied an inverse boundary value problem for the magnetic Schr\"odinger equation in the presence of obstacles whose boundaries were not available for measurements.
He reduced unique determination of the electromagnetic potential from partial Cauchy data to injectivity of the broken ray transform.

Similarly, Kenig and Salo~\cite{KS:calderon} reduced a partial data problem for the conductivity equation to the broken ray transform.
Both results are based on constructing solutions to the PDE that concentrate near broken rays.
We believe that injectivity results for the broken ray transform will find applications in other inverse boundary value problems as well.

On a slightly different note, one may also take the entire boundary of the domain (or manifold) to be reflective ($E=\emptyset$) and consider periodic broken rays.
The natural question in this setting asks whether a function is determined by its integral over all periodic broken rays, that is, whether the periodic broken ray transform is injective.
The periodic broken ray transform was discussed by the first author in~\cite{I:refl}, where examples and counterexamples to injectivity were given.
The regularity assumptions under which the periodic broken ray transform on a rectangular domain is injective have been subsequently weakened significantly~\cite{I:torus}.

As will be shown below in Theorem~\ref{thm:lin}, linearizing lengths of broken rays with respect to the metric gives rise to the broken ray transform.
It seems plausible that linearizing lengths of periodic broken rays may lead to the periodic broken ray transform in a similar fashion.
Lengths of periodic broken rays are intimately related to spectral geometry; see for example~\cite{DH:spectral-survey}.


\subsection{Boundary distance function and lens data}
\label{sec:lens}

It is well known (see e.g.~\cite[Section~1.1]{S:tensor-book}) that the geodesic X-ray transform arises in linearized versions of the boundary rigidity problem.
In this section we show that the broken ray transform arises as a linearization in a similar manner.

The linearization leads generally to the broken ray transform of symmetric tensor fields on Riemannian manifolds.
Theorem~\ref{thm:brt} only solves the linearized problem in a special case when the metric is varied within a conformal class on a Riemannian surface.

Let~$M$ be a manifold with boundary, with its boundary divided in disjoint parts~$E$ and~$R$.
Given a point on~$E$ and and an inward unit vector, a Riemannian metric on~$M$ determines a unique broken ray.
Assuming this broken ray eventually hits~$E$ again, we can map the starting point and direction into the exit point and direction; this map is the broken ray scattering relation.
(This map is different from the broken scattering relation used in~\cite{KLU:broken-geodesic-flow}.)
We can also map the initial point and direction into the travel time.

A natural question regarding these maps is the following:
Does the broken ray scattering relation, the travel time map, or both of these together determine the metric up to isometry?

This question can also be linearized.
Let~$g_s$, $s\in(a,b)\subset\R$, be a one parameter family of Riemannian metrics on~$M$.
The linearized question is:
Does the derivative of one or both of the two maps (scattering relation and travel time) with respect to~$s$ determine the derivative~$\Der{s}g_s$?
The derivative~$\Der{s}g_s$ is in general a symmetric 2-tensor, but if the metric is only varied within a conformal class, it can be viewed as a function.

Studying these questions may have applications in geophysics.
For example, in the event of an earthquake one may measure the arrival time of seismic waves to other points on the planet's surface.
Given such data for multiple earthquakes, one would like to infer the interior structure of the Earth.
This structure is described by a Riemannian metric (corresponding to a possibly anisotropic wave speed) so that seismic waves travel along (unit speed) geodesics.
Measuring arrival times from earthquake sites all around the Earth thus corresponds to measuring the lengths of all geodesics.

In practice the situation is not as simple, partly due to reflections.
Seismic waves (partly) reflect from the core and also on the surface.
Since seismic measurements are difficult to do in oceans, even some seismic waves near the surface can only be measured after they have reflected -- possibly several times.
Studying arrival times of reflected seismic waves is thus closely related to studying lengths of broken rays.

An answer to the linearized question is given by the following theorem.

\begin{theorem}
\label{thm:lin}
Let $(M,g)$ be a manifold with boundary whose boundary is split in disjoint parts~$E$ and~$R$ and let $\eps>0$.
Suppose~$g_s$, $s\in(-\eps,\eps)$, is a family of Riemannian metrics on~$M$ depending $C^1$-smoothly on~$s$ such that $g_0=g$.
Let $(x_0,v_0)$ be an inward pointing unit vector based in~$E$ and $\gamma_{(x_0,v_0),g_s}$ be the broken ray starting at it with respect to the metric~$g_s$.
Denote the length of this broken ray by $\tau_{(x_0,v_0),g_s}$, so that $\gamma_{(x_0,v_0),g_s}:[0,\tau_{(x_0,v_0),g_s}]\to M$.

If the points where $\gamma_{(x_0,v_0),g_s}$ reflects on~$R$ and where it exits~$E$ are not in $\bar R\cap\bar E$ and they meet~$\partial M$ non-tangentially, then
\begin{equation}
\begin{split}
&
2\int_0^{\tau_{(x_0,v_0),g_0}}
\left.\Der{s}\abs{\dot\gamma_{(x_0,v_0),g_0}(t),\dot\gamma_{(x_0,v_0),g_0}(t)}^2_{g_s}\right|_{s=0}
\der t
\\&=
\left.\Der{s}\tau_{(x_0,v_0),g_s}\right|_{s=0}
\\&\quad-
\ip{\dot\gamma_{(x_0,v_0),g_0}(\tau_{(x_0,v_0),g_0})}{\left.\Der{s}\gamma_{(x_0,v_0),g_s}(\tau_{(x_0,v_0),g_s})\right|_{s=0}}_{g_0}.
\end{split}
\end{equation}
In particular the derivative of the travel time map and the scattering relation with respect to the parameter~$s$ determine the broken ray transform of the symmetric 2-tensor $\Der{s}g_s|_{s=0}$ on such broken rays.

Actually, the claim is true whenever $\gamma_{(x_0,v_0),g_s}$ is a~$C^1$ curve which depends $C^1$-smoothly on~$s$ and converges uniformly to $\gamma_{(x_0,v_0),g_0}$ as $s\to0$.
In particular, the curves $\gamma_{(x_0,v_0),g_s}$ may be replaced with the broken rays with respect to~$g_s$ that connect the endpoints of $\gamma_{(x_0,v_0),g_0}$ and have the same reflection pattern.
\end{theorem}

We do not need to assume that the reflection map~$\rho$ is the same for all metrics~$g_s$.
Also, we do not need the full scattering map, but only the endpoint map.
If $g_s=f_sg_0$ for some family of scalar functions~$f_s$ with $f_0\equiv1$, the integral of $\Der{s}g_s|_{s=0}$ becomes the integral of the scalar function $\Der{s}f_s|_{s=0}$.
The proof of this theorem follows quite simply from the following lemma.

\begin{lemma}
\label{lma:lin}
Let $(M,g)$ be a manifold with boundary with boundary and let $\eps>0$.
Suppose~$g_s$, $s\in(-\eps,\eps)$, is a family of Riemannian metrics on~$M$ depending $C^1$-smoothly on~$s$ such that $g_0=g$.
Let $x_s\in\partial M$ be a point and $v_s\in T_{x_s}M$ an inward unit vector varying $C^1$-smoothly with~$s$.

Denote the corresponding geodesic by $\gamma^1_{(x_s,v_s),g_s}$ and its length by $\tau_{(x_s,v_s),g_s}$.
If the velocity $\dot\gamma^1_{(x_0,v_0),g_0}(\tau^1_{(x_0,v_0),g_0})$ is non-tangential, then
\begin{equation}
\label{eq:geodesic-lin}
\begin{split}
&
2\int_0^{\tau^1_{(x_0,v_0),g_0}}
\left.\Der{s}\abs{\dot\gamma^1_{(x_0,v_0),g_0}(t),\dot\gamma^1_{(x_0,v_0),g_0}(t)}^2_{g_s}\right|_{s=0}
\der t
\\&=
\left.\Der{s}\tau^1_{(x_s,v_s),g_s}\right|_{s=0}
\\&\quad+
\ip{v_0}{\left.\Der{s}x_s\right|_{s=0}}_{g_0}
\\&\quad-
\ip{\dot\gamma^1_{(x_0,v_0),g_0}(\tau^1_{(x_0,v_0),g_0})}{\left.\Der{s}\gamma^1_{(x_s,v_s),g_s}(\tau^1_{(x_s,v_s),g_s})\right|_{s=0}}_{g_0}.
\end{split}
\end{equation}
\end{lemma}

\begin{proof}
To shorten notations, we denote $L_s=\tau_{(x_s,v_s),g_s}$, $L'_s=\Der{s}L_s$, $\gamma_s=\gamma^1_{(x_s,v_s),g_s}$, $\gamma'_s=\Der{s}\gamma_s$, and $f=\Der{s}g_s|_{s=0}$.
We rescale the perturbed geodesics so that they are all parametrized by the same interval $[0,L_0]$ by letting $\bar\gamma_s(t)=\gamma_s(tL_s/L_0)$.
We also let $\bar\gamma_s'=\Der{s}\bar\gamma_s$.
We make calculations in local coordinates as if all geodesics~$\gamma_s$ were contained in the same coordinate patch; the results from different patches can be combined easily.

A simple calculation gives
\begin{equation}
\begin{split}
2L'_0
&=
\Der{s}(L_s^2/L_0)|_{s=0}
\\&=
\left.\Der{s}
\int_0^{L_0}\abs{\dot{\bar\gamma}_s(t)}^2_{g_s}\der t
\right|_{s=0}
\\&=
\left.\Der{s}
\int_0^{L_0}(g_s)_{ij}(\bar\gamma_s(t))\dot{\bar\gamma}^i_s(t)\dot{\bar\gamma}^j_s(t)\der t
\right|_{s=0}
\\&=
\int_0^{L_0} f_{ij}(\bar\gamma_0(t))\dot{\bar\gamma}^i_0(t)\dot{\bar\gamma}^j_0(t)\der t
\\&\quad+
\int_0^{L_0}\partial_k(g_0)_{ij}(\bar\gamma_0(t))(\bar\gamma'_0)^k(t)\dot{\bar\gamma}^i_0(t)\dot{\bar\gamma}^j_0(t)\der t
\\&\quad+
2\int_0^{L_0}(g_0)_{ik}(\bar\gamma_0(t))\dot{\bar\gamma}^i_0(t)(\dot{\bar\gamma}'_0)^k(t)\der t.
\end{split}
\end{equation}
Integrating by parts in the last integral and noting that $\bar\gamma_0=\gamma_0$ we get
\begin{equation}
\begin{split}
2L'_0
&=
\int_0^{L_0} f_{ij}(\gamma_0(t))\dot\gamma^i_0(t)\dot\gamma^j_0(t)\der t
\\&\quad+
\int_0^{L_0}(\bar\gamma'_0)^k(t)
\left[
\partial_k(g_0)_{ij}(\gamma_0(t))\dot\gamma^i_0(t)\dot\gamma^j_0(t)
-
2\Der{t}((g_0)_{ik}(\gamma_0(t))\dot\gamma^i_0(t))
\right]\der t
\\&\quad+
2\left[
(g_0)_{ik}(\gamma_0(L_0))\dot\gamma^i_0(L_0)(\bar\gamma'_0)^k(L_0)
-
(g_0)_{ik}(\gamma_0(0))\dot\gamma^i_0(0)(\bar\gamma'_0)^k(0)
\right].
\end{split}
\end{equation}
Now $\partial_k(g_0)_{ij}(\gamma_0(t))\dot\gamma^i_0(t)\dot\gamma^j_0(t)-2\Der{t}((g_0)_{ik}(\gamma_0(t))\dot\gamma^i_0(t))=0$ since $\gamma_0$ solves the geodesic equation.
We have also $\bar\gamma'_0(0)=\Der{s}\gamma_s(0)|_{s=0}$ and $\bar\gamma'_0(L_0)=\Der{s}\gamma_s(L_s)|_{s=0}$, so
\begin{equation}
\begin{split}
2L'_0
&=
\int_0^{L_0} f_{ij}(\gamma_0(t))\dot\gamma^i_0(t)\dot\gamma^j_0(t)\der t
\\&\quad+
2\left[
\ip{\dot\gamma_0(L_0)}{\Der{s}\gamma_s(L_s)|_{s=0}}_{g_0}
-
\ip{\dot\gamma_0(0)}{\Der{s}\gamma_s(0)|_{s=0}}_{g_0}
\right].
\end{split}
\end{equation}
This is exactly equation~\eqref{eq:geodesic-lin}.
\end{proof}

\begin{proof}[Proof of Theorem~\ref{thm:lin}]
The broken ray $\gamma_s\coloneqq\gamma_{(x_0,v_0),g_s}$ is composed of $N\in\N$ geodesic segments~$\gamma_{s,m}$, $m\in\{1,\dots,N\}$.
Since~$\gamma_0$ has no tangential reflections, each~$\gamma_s$, $s\in(-\eps,\eps)$, has the same number of reflections if~$\eps$ is sufficiently small.

Denote the lengths of the broken rays by $\tau_s=\tau_{(x_0,v_0),g_s}$ and the length of the segments~$\gamma_{s,m}$ by~$\tau_{s,m}$.
Also, let~$p_{s,m}^\pm$ be the final ($+$) and initial ($-$) points of the segments~$\gamma_{s,m}$.
Then, using $p_{s,m}^+=p_{s,m+1}^-$, $\Der{s}p_{s,1}^-|_{s=0}=0$, and Lemma~\ref{lma:lin}, we get
\begin{equation}
\begin{split}
\Der{s}\tau_s|_{s=0}
&=
\sum_{m=1}^N\Der{s}\tau_{s,m}|_{s=0}
\\&=
2\sum_{m=1}^N\int_0^{\tau_{0,m}}\left.\Der{s}\abs{\dot\gamma_{0,m}(t)}^2_{g_s}\right|_{s=0}\der t
\\&\quad+
\sum_{m=1}^N\left[
\ip{\dot\gamma_{0,m}(\tau_{0,m})}{\Der{s}p_{s,m}^+|_{s=0}}
-
\ip{\dot\gamma_{0,m}(0)}{\Der{s}p_{s,m}^-|_{s=0}}
\right]
\\&=
2\int_0^{\tau_0}\left.\Der{s}\abs{\dot\gamma_0(t)}^2_{g_s}\right|_{s=0}\der t
\\&\quad+
\ip{\dot\gamma_0(\tau_0)}{\Der{s}p_{0,m}^+|_{s=0}}.
\end{split}
\end{equation}
This identity is precisely the claim.
\end{proof}

\begin{remark}
A given broken ray is not necessarily the shortest broken ray joining its endpoints.
Linearizing the boundary distance function does not give the integral of the variation of the metric over all broken rays, but only those that minize the length.
If we knew that the set of nonminimizing broken rays is somehow small (on~$E_-$ for example), we could recover the whole broken ray transform, but we do not know if it is small in general.
In a Euclidean domain with one strictly convex obstacle a broken ray minizes length if and only if it has no reflections or only a tangential one.
\end{remark}

\bibliographystyle{abbrv}
\bibliography{brt}

\end{document}